\documentclass[graybox]{svmult}

\usepackage{mathptmx}       % selects Times Roman as basic font
\usepackage{helvet}         % selects Helvetica as sans-serif font
\usepackage{courier}        % selects Courier as typewriter font
\usepackage{graphicx}        % standard LaTeX graphics tool
\usepackage{multicol}        % used for the two-column index
\usepackage[bottom]{footmisc}% places footnotes at page bottom

\usepackage{cite}
\usepackage{fancyvrb}        % LaTeX tool for writing verbatim
\usepackage{tikz}
\usepackage{graphics}
\usepackage{float}

\usepackage{subcaption}
\usepackage{pgfplots}

\usepackage{amsmath}
\newcommand{\dt}{\partial_t\,}
\newcommand{\dtau}{d_\tau\,}
\newcommand{\mytop}{\mathsf{T}}
\newcommand{\bb}{\mathbf{b}}
\newcommand{\hh}{\mathbf{h}}
\newcommand{\ee}{\mathbf{e}}
\newcommand{\dd}{\mathbf{d}}
\newcommand{\pp}{\mathbf{p}}
\newcommand{\eps}{\epsilon}
\newcommand{\jj}{j}
\newcommand{\Mh}{\mathsf{M_h}}
\newcommand{\Me}{\mathsf{M_e}}
\newcommand{\Mpi}{\mathsf{M_{p,i}}}
\newcommand{\Mdi}{\mathsf{M_{d,i}}}
\newcommand{\C}{\mathsf{C}}
\newcommand{\M}{\mathsf{M}}

\newcommand{\curl}{\operatorname{curl}}

\begin{document}

\title*{A convolution quadrature method for Maxwell's equations in dispersive media}
\author{J{\"u}rgen D{\"o}lz 
\and
Herbert Egger
\and
Vsevolod Shashkov 
}
\institute{J{\"u}rgen D{\"o}lz \at University of Twente,
  \email{j.dolz@utwente.nl}
\and
Herbert Egger, Vsevolod Shashkov \at TU Darmstadt, 
\email{herbert.egger@tu-darmstadt.de, shashkov@mathematik.tu-darmstadt.de}}

\maketitle

\abstract*{We study the systematic numerical approximation of Maxwell's equations in dispersive media. Two discretization strategies are considered, one based on a traditional leapfrog time integration method and the other based on convolution quadrature. The two schemes are proven to be equivalent and to preserve the underlying energy-dissipation structure of the problem. The second approach, however, is independent of the number of internal states and allows to handle rather general dispersive materials. Using ideas of fast-and-oblivious convolution quadrature, the method can be implemented efficiently.}

\abstract{We study the systematic numerical approximation of Maxwell's equations in dispersive media. Two discretization strategies are considered, one based on a traditional leapfrog time integration method and the other based on convolution quadrature. The two schemes are proven to be equivalent and to preserve the underlying energy-dissipation structure of the problem. The second approach, however, is independent of the number of internal states and allows to handle rather general dispersive materials. Using ideas of fast-and-oblivious convolution quadrature, the method can be implemented efficiently.\\}

\section{Introduction} \label{sec:1}

We consider electromagnetic wave propagation through linear dispersive media.  The underlying physics are described by Maxwell's equations
\begin{align} \label{eq:max1}
\dt \dd &= \curl \hh, \qquad \dt \bb = -\curl \ee 
\end{align}
with $\ee,\hh$ and $\dd,\bb$ denoting the electric and magnetic fields and fluxes, respectively, which are mutually related by the constitutive relations
\begin{align} \label{eq:max2}
\bb = \mu_0 \hh, \qquad \dd = \eps_0 \eps_\infty \ee + \pp. 
\end{align}
Here $\epsilon_0, \mu_0$ are the permittivity and permeability of vacuum, and $\eps_\infty=1+\eps_\infty'$ is the high frequency limit of the relative permittivity.
Further, $\pp$ denotes the memory part of the polarization $\pp_{tot}=\eps_0 \eps_\infty' \ee + \pp$, 
which is described in frequency domain by 
\begin{align} \label{eq:max3}
  \hat \pp(s) = \eps_0 \hat \chi(s)\hat \ee(s).
\end{align}
The system is complemented by appropriate boundary and initial conditions. For ease of presentation, we assume that $\ee(0)=\pp(0)=0$ in the following.
By inverse Laplace-transform, the polarization can then be expressed in time domain by 
\begin{align} \label{eq:max4}
\pp(t) = \eps_0 \int_0^t \chi(t-s) \ee(s) ds.
\end{align}
We further assume throughout the paper that the susceptibility kernel $\chi$ can be written as a superposition of simple Debye functions \cite{Debye29}, i.e., 
\begin{equation} \label{eq:debye}    
\hat \chi(s) = \sum\nolimits_{i} \hat \chi_i(s) 
\qquad \text{with} \qquad 
\hat\chi_i(s) = \frac{\epsilon_{i,s}-\epsilon_{i,\infty}'}{1 + s \tau_i},
\end{equation}
where $\tau_i$ denotes the relaxation time and $\eps_{i,s}$, $\eps_{i,\infty}'$ are the static and high-frequency limits of the electric susceptibility of the $i$th component with $\sum_i \eps_{i,\infty}' = \eps_\infty'$. 
Such multipole Debye models have been used, e.g., for the modeling of the dielectric response of biological tissue; see \cite{Clegg12,Gabriel96} and the references given there. 
In general, the summation in \eqref{eq:debye} may be over infinitely many terms. % 

One of the key features of the multipole Debye model is its provable passivity, which follows from the energy--dissipation principle \cite{Bokil14,Lanteri12} 
\begin{equation}
\frac{d}{dt}\mathcal{E} = -\sum\nolimits_i \| \sqrt{\tfrac{\tau_i}{\epsilon_0 (\epsilon_{i,s}- \epsilon_{i,\infty}')}} \dt\pp_i\|^2, 
\end{equation}
valid for any sufficiently smooth solution of \eqref{eq:max1}--\eqref{eq:max3} with homogeneous or periodic boundary conditions. Here $\|\cdot\|$ is the $L^2$-norm, further $\pp = \sum_i \pp_i$ is the decomposition of the memory part of the polarization into its components according to \eqref{eq:debye}, and  
\begin{equation}
\label{shashkov:eq:energy}
\mathcal{E} = \frac{1}{2}\left( \| \sqrt{\mu_0}\hh \|^2 + \|\sqrt{\epsilon_0\epsilon_\infty} \ee \|^2 + \sum\nolimits_i \| \tfrac{1}{\sqrt{\epsilon_0(\epsilon_i-\epsilon_\infty')}}\pp_i \|^2 \right)
\end{equation}
denotes the electromagnetic energy of the system. 
Due to the rational structure of the transfer functions $\hat\chi_i$, the individual polarizations $\pp_i$ can be characterized equivalently by the differential equations 
\begin{align} \label{eq:debye_ode}
\tau_i \dt \pp_i + \pp_i = \eps_0 (\epsilon_{i,s}-\epsilon_{i,\infty}') \ee,
\end{align}
with initial values $\pp_i(0) =0$, which is the basis for various simulation methods. 
Corresponding finite difference and finite element schemes have been considered, for instance, in \cite{Bokil14,Gandhi93,Jenkinson18,Jiao01,Lanteri12,Li06,Luebbers90,Shaw10}. 
Let us note that with increasing number of internal states $\pp_i$, all methods become computationally more and more expensive.  

In this paper, we consider a different approach for the numerical solution of \eqref{eq:max1}--\eqref{eq:max3}, which allows to us compute the time evolution of $\ee$, $\hh$, and $\pp$ without explicitly computing the internal states $\pp_i$. 
As indicated in \cite{Egger20}, this can be accomplished through discretization of the integral \eqref{eq:max4} by means of appropriate convolution quadratures \cite{Lubich88a,Lubich93}, instead of integrating \eqref{eq:debye_ode} with time-differencing schemes. 
The complexity of every time step is then independent of the number of internal states $\pp_i$. Moreover, using ideas of \cite{Roychowdhury99,Schaedle05}, the additional memory cost for storing the history of the field $\ee$ can be reduced to the logarithm of the number of time steps. 

The remainder of the manuscript is organized as follows: 
In Section~\ref{sec:2}, we briefly discuss the discretization of \eqref{eq:max1}--\eqref{eq:max2} together with \eqref{eq:debye_ode} by means of standard methods, and we present a short proof of the underlying energy-dissipation structure, which results in passivity of the discrete scheme. %   
In Section~\ref{sec:3}, we then introduce our alternative approach based on convolution quadrature, and we prove its equivalence with the method discussed in Section~\ref{sec:2}. As a consequence, the favorable stability properties of standard schemes are automatically inherited. 
In~Section~\ref{sec:4}, we present computational results for the propagation of an electromagnetic pulse across the interface between air and human tissue and we illustrate the energy--dissipation behavior as well as the equivalence of the two schemes discussed in the paper. 

\vspace*{-1em}

\section{Structure preserving discretization} \label{sec:2}

After space discretization by appropriate finite-difference or finite-element methods and time-discretization by the leapfrog scheme, the system \eqref{eq:max1}--\eqref{eq:max2} with polarization components defined by \eqref{eq:debye_ode} can be written in
matrix--vector notation as 
\begin{alignat}{5}
\Mh \, \dtau \hh^{n} + \C \, \ee^{n} &= 0, \label{eq:disc1} \\
\Me \, \dtau \ee^{n+1/2} + \sum\nolimits_i \dtau \pp_i^{n+1/2} - \C^\mytop \, \overline \hh^{n} &= 0, \label{eq:disc2}\\
\Mdi \, \dtau \pp_i^{n+1/2} + \Mpi \, \overline \pp_i^{n+1/2} &= \overline \ee^{n+1/2}, \qquad i \ge 1. \label{eq:disc3} 
\end{alignat}
The equations hold for all $n \ge 0$ and are complemented by appropriate initial conditions.
Here $u^n$ and $u^{n+1/2}$ are the approximations for $u(t^n)$ and $u(t^{n+1/2})$ with $t^s=s \tau$ and $\tau$ denoting the time step size. 
Furthermore, $\dtau u^n = \frac{1}{\tau} (u^{n+1/2} - u^{n-1/2})$ is the central difference quotient, and $\overline u^n = \frac{1}{2}(u^{n+1/2}+u^{n-1/2})$ the average of two steps.
Note that equation~\eqref{eq:disc3} was obtained from \eqref{eq:debye_ode} after dividing by $\eps_0 (\eps_{i,s} - \eps_{i,\infty}')$. 

For appropriate space discretization schemes, the mass matrices $\M_*$ are symmetric, positive-definite, and diagonal or block-diagonal \cite{Cohen02,Egger18}, such that \eqref{eq:disc1}--\eqref{eq:disc3} amounts to an explicit time-stepping scheme.
Moreover, the method satisfies the following discrete equivalent of the underlying energy--dissipation identity.
\begin{lemma} \label{lem:1}
Set $\|a\|_\M^2=(a,a)_\M$ and $(a,b)_\M=b^\mytop \M a$, and denote by
\begin{align*}
\mathcal{E}^{n} = \frac{1}{2} \left( (\hh^{n+1/2},\hh^{n-1/2})_{\Mh} + \|\ee^n\|^2_{\Me} + \sum\nolimits_i \|\pp_i^n\|^2_{\Mpi} \right)
\end{align*}
the discrete energy at time step $t^n=n\tau$. 
Then any solution of \eqref{eq:disc1}--\eqref{eq:disc3} satisfies
\begin{align*}
\dtau \mathcal{E}^{n+1/2} = - \sum\nolimits_i \| \dtau \pp_i^{n+1/2} \|^2_{\Mdi}, \qquad n \ge 0.
\end{align*}
\end{lemma}
\begin{proof}
By elementary computations, one can verify that 
\begin{align*}
\dtau \mathcal{E}^{n+1/2} 
&= \frac{1}{2}(\dtau \hh^{n+1} + \dtau \hh^{n},\hh^{n+1/2})_{\Mh} 
 + (\dtau \ee^{n+1/2},\overline \ee^{n+1/2})_{\Me}\\
& \qquad \qquad \quad \   + \sum\nolimits_i (\dtau \pp_i^{n+1/2}, \overline \pp_i^{n+1/2})_{\Mpi}. 
\end{align*}
Note that $(a,b)_M = (M a,b) = (M b,a)$ where $(\cdot,\cdot)$ denotes the Euclidean scalar product. We then test equation \eqref{eq:disc2} with $\overline\ee^{n+1/2}$ and \eqref{eq:disc3} with $\dtau \pp^{n+1/2}$. Moreover, we test the average of equation \eqref{eq:disc1} for step $n$ and $n+1$ with $\hh^{n+1/2}$. This allows to replace all terms on the right hand side of the above formula and leads to 
\begin{align*}
\dtau \mathcal{E}^{n+1/2} 
&= -(\C \overline \ee^{n+1/2}, \hh^{n+1/2}) + (\C^\mytop \hh^{n+1/2} - \sum\nolimits_i \dtau \pp_i^{n+1/2},\overline \ee^{n+1/2}) \\
& \qquad \qquad \quad  + \sum\nolimits_i (\overline \ee^{n+1/2} - \Mdi \dtau \pp_i^{n+1/2}, \dtau \pp_i^{n+1/2}). 
\end{align*}
Using that $(\C a,b)=(\C^\mytop b,a)$, one can see that most of the terms drop out 
and we obtain the assertion of the lemma. \hfill $ $ \qed
\end{proof}
\begin{remark}
Method \eqref{eq:disc1}--\eqref{eq:disc3} automatically inherits the energy-dissipation principle of the continuous problem. We therefore call it a \emph{structure-preserving} discretization scheme.
The first term in the energy $\mathcal{E}$ can be estimated from below by 
\begin{align*}
(\hh^{k+1/2},\hh^{k-1/2})_{\Mh} 
&= \|\hh^{k+1/2}\|^2_{\Mh} + \tau (\hh^{k+1/2},\dtau \hh^{k})_{\Mh} \\ 
&= \|\hh^{k+1/2}\|^2_{\Mh} - \tau (\C \ee^k,\hh^{k}) \\
&\ge \frac{1}{2} \|\hh^{k+1/2}\|^2_{\Mh} - \frac{\tau^2}{2} \|\C \ee^k\|^2_{\Mh^{-1}},
\end{align*}
and the last term can be further bounded from below 
under the assumption that 
\begin{align} \label{eq:cfl}
\tau^2 \|\C \ee\|^2_{\Mh^{-1}} \le \|\ee\|_{\Me}^2.
\end{align}
This standard CFL condition for the leapfrog method implies stability of the scheme and allows to show that the energy $\mathcal{E}$ is a positive and symmetric quadratic functional and thus induces a norm on the space of state vectors $(\hh,\ee,\pp_1,\pp_2,\ldots)$. 
Together with Lemma~\ref{lem:1}, this is the basis for the error analysis of method \eqref{eq:disc1}--\eqref{eq:disc3}; we refer to \cite{Joly03} for details. 
\end{remark}

\section{A convolution quadrature approach} \label{sec:3}

The dimension of the state space and hence also the computational cost for computing one time step of method \eqref{eq:disc1}--\eqref{eq:disc3} obviously increases with increasing number of internal states $\pp_i$. We will now show that $\ee$, $\hh$, and $\pp=\sum_i \pp_i$ can be computed without explicit reference to the internal states $\pp_i$, which results in an algorithm that is \emph{independent of the number of internal states}.
Instead of using equation \eqref{eq:debye_ode}, we directly discretize the integral \eqref{eq:max4} by a convolution sum
\begin{align} \label{eq:convsum}
\pp^n = \sum\nolimits_{k=0}^n \omega_{n-k}\ee^k.
\end{align}
This is the field of convolution quadrature, and we refer to \cite{Lubich88a,Lubich93} for details on the mathematical background.
As illustrated in \cite{Egger20}, a proper choice of the convolution weights $\{\omega_n\}_{n\ge 0}$ allows to obtain the following equivalence statement. 
\begin{lemma} \label{lem:2}
Let $\{\omega_n\}_{n \ge 0}$ be the coefficients of the power series 
\begin{align}\label{eq:weights}
\eps_0 \hat\chi \left( \tfrac{2(1-\xi)}{ \tau(1+\xi)}\right) = \sum_{n=0}^\infty\omega_n\xi^n.
\end{align}
Then the solution $\{\hh^{n+1/2},\ee^n, \pp^n\}_{n \ge 0}$ of the scheme
\eqref{eq:disc1}--\eqref{eq:disc3} with $\ee^0=\pp_i^0=0$ 
coincides with the solution of the convolution-quadrature method \eqref{eq:disc1}--\eqref{eq:disc2} and \eqref{eq:convsum}. 
\end{lemma}
\begin{proof}
For convenience of the reader, we briefly summarize the basic ideas of the proof, which closely follows the arguments presented in \cite{Egger20}. 
We start by multiplying equations \eqref{eq:disc3} with $\xi^n$ 
and sum over all $n \ge 0$ to obtain  
\begin{align*}
\sum\nolimits_{n \ge 0} \Mdi \, (\tfrac{1}{\xi} -1)  \pp_i^n \xi^n 
+ \sum\nolimits_{n \ge 0} \Mpi \, (\tfrac{1}{2\xi}+\tfrac{1}{2}) \pp_i^n \xi^n 
&= \sum\nolimits_{n \ge 0} (\tfrac{1}{2\xi}+\tfrac{1}{2}) \ee^n \xi^n. 
\end{align*}
An appropriate rearrangement of terms then further leads to 
\begin{align*}
\sum\nolimits_{n \ge 0} \pp_i^n \xi^n  = \hat\chi_i \left(\tfrac{2(1-\xi)}{\tau(1+\xi)}\right)\sum\nolimits_{n\ge0} \ee^n \xi^n,
\end{align*}
with transfer function $\hat \chi_i$ as defined in \eqref{eq:debye}.    
Summation over all $i$ and using $\pp^n=\sum_i \pp^n_i$ and the definition of the weights $\omega_n$ then yields the assertion. \hfill $ $ \qed
\end{proof}
\begin{remark} \label{rem:cq}
According to the above lemma, the convolution quadrature (CQ) method defined by \eqref{eq:disc1}--\eqref{eq:disc2} and \eqref{eq:convsum}--\eqref{eq:weights} has the same passivity and stability properties as the underlying difference scheme \eqref{eq:disc1}--\eqref{eq:disc3}. 
Let us note that instead of the internal states $\{\pp_i^n\}_{i \ge 0}$, the CQ approach utilizes the history $\{\ee^k\}_{k \le n}$ of the electric field values to compute the memory part $\pp^n$ of the polarization. 
\end{remark}

Before closing this section, we briefly comment on the practical computation of the weights $\{\omega_n\}_{n \ge 0}$ and the efficient realization of the proposed CQ approach. 

\begin{remark} \label{rem:weights}
Following \cite{Lubich88a,Lubich88b}, also see \cite{Egger20}, the convolution weights $\{\omega_n\}_{n\ge0}$ can be computed with high accuracy using fast Fourier transforms, i.e., 
\begin{align*} 
\omega_n \approx \tfrac{1}{L\rho^n} \sum\nolimits_{\ell=0}^{L-1} \hat \chi \left(\tfrac{2}{\tau} \tfrac{1-\rho e^{\imag\phi_\ell}}{1+\rho e^{\imag\phi_\ell}}\right) e^{-\imag n\phi_\ell}, \qquad \phi_\ell = 2\pi \ell/L,
\end{align*}
and the quadrature error can be controlled by appropriate choice of the parameters $L$ and $\rho$; see \cite{Lubich88a,Lubich88b,Lubich93} for details. 
The computation of all weights $\{\omega_{n}\}_{n=0}^N$ with machine precision requires $O(N)$ evaluations of $\hat \chi$.
If the material parameters are inhomogeneous, then the weights $\omega_n$ will also depend on the spatial variable. 
\end{remark}

\begin{remark} \label{rem:focq}
A straight-forward implementation of the CQ approach would require the storage of the complete history $\{\ee^k\}_{k \le n}$ to compute the polarization $p^n$ via \eqref{eq:convsum}. 
By ideas of \cite{Roychowdhury99,Schaedle05}, the required storage can be reduced to $O(\log N)$ field vectors, leading to a \emph{fast-and-oblivious} convolution quadrature (FOCQ) method. 
The basic idea is to divide the convolution sum \eqref{eq:convsum} 
into exponentially growing subsums
\begin{align*}
\sum\nolimits_{k=0}^n \omega_{k}\ee^{n-k} = \sum\nolimits_{\ell=0}^{L} \sum\nolimits_{k = B^{\ell-1}-1}^{B^\ell} \omega_{k}\ee^{n-k} =:\sum\nolimits_{\ell=0}^{L} U_n^\ell,
\end{align*}
where $B>1$ is an integer and we assumed for simplicity that $n=B^L$ is a power of the basis $B$; otherwise the first few summands are taken into account separately. 
Under certain regularity assumptions on $\hat \chi$, each subsum $U_n^\ell$ can be approximated efficiently using interpolation \cite{Roychowdhury99} or contour integration techniques \cite{Schaedle05}.
The FOCQ algorithm thus requires only to store $O(\log N)$ historical field vectors and further only $O(\log N)$ evaluations of the transfer function $\hat\chi$ are needed.
\end{remark}

\section{Numerical illustration} \label{sec:4}

In our test problem, we consider the propagation of an electromagnetic pulse across the interface between air and human tissue. The dielectric response of the tissue 
is characterized by a five-pole Debye model which was taken from \cite{Gandhi93}. 
Using the notation of Section~\ref{sec:1}, the total polarization in this model is prescribed in frequency domain by 
$\hat \pp_{tot}(s)=\eps_0 (\eps_\infty' + \hat\chi(s)) \hat \ee(s)$ with $\eps_\infty'=3.3$ and 
\begin{align*}
\hat\chi(\jj\omega) 
= \frac{8.5\cdot 10^5}{1+ \jj\omega/(138\pi)}
&+\frac{8.19 \cdot 10^3}{1+ \jj\omega/(86\pi\cdot 10^3)} 
+\frac{1.19 \cdot 10^3}{1+ \jj\omega/(1.34\pi\cdot10^6)}  \\
& +\frac{32}{1+ \jj\omega/(460\pi\cdot 10^6)}
+\frac{45.8}{1+ \jj\omega/(40\pi\cdot 10^9)}.  
\end{align*}
For our computational tests, we consider a plane wave setting, in which the fields are of the form $\ee = (e_x,0,0)$, $\hh=(0,h_y,0)$, and $\pp_i = (p_{x,i},0,0)$, and only depend time $t$ and the propagation direction $z$. Then \eqref{eq:max1}--\eqref{eq:max4} leads to a one--dimensional wave propagation problem for unknown fields $e_x$, $p_x$ and $h_y$.  
As computational domain, we consider the interval $(-1,1)$ and we impose periodic boundary conditions for the electric and magnetic field. 
The initial values are described by $e_{x,0}(z) = p_{x,i,0}(z)= 0$ and $h_{y,0}(z)=10 e^{-10 z^2}$.
All quantities are given in SI-units.

For the spatial discretization, we utilize  piecewise linear finite elements for $e_x$ and $p_{x,i}$, and piecewise constants to represent $h_y$. Numerical integration by the vertex rule is used for the assembling of the mass matrices $\Me$, $\Mpi$, and $\Mdi$, which leads to a diagonal structure, and the matrix $\Mh$ is diagonal automatically. 
In Figure~\ref{fig:wave}, we display the magnetic field component $h_y$ for the two schemes presented in Section~\ref{sec:2} and \ref{sec:3} for some selected time steps. 
As predicted, the numerical solutions cannot be distinguished by visual inspection; in our computations, the maximal difference was in the order of $10^{-12}$ and thus much smaller than the discretization errors. 
In our computations we tested both, the classical CQ and the FOCQ approach, leading to almost identical results. The latter was substantially faster, in particular for a large numbers of time steps.
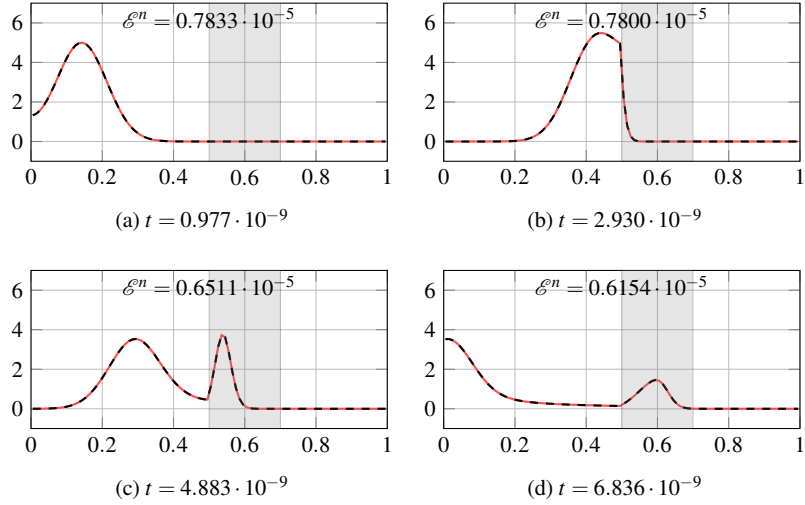
\begin{figure}[ht]
\centering
\begin{subfigure}{0.45\textwidth}
\centering
\begin{tikzpicture}
\begin{axis}[width=1.2\textwidth, height = 0.7\textwidth,ymin=-1,ymax=7,xmin=0,xmax=1,grid=both]
\addplot[mark=none, red!70, thick] table [y=y100,x=x, col sep=comma]{leapfrog_4.txt};
\addplot[mark=none, black, dashed, thick] table [y=y100,x=x, col sep=comma]{conv_4.txt};
\draw[mark=none, fill = gray, opacity=0.2] (axis cs:0.5,-1) rectangle (axis cs:0.7,8);
\draw (50,725) node { $\mathcal{E}^n = 0.7833 \cdot 10^{-5}$};
\end{axis}
\end{tikzpicture}
\caption{$t = 0.977\cdot 10^{-9}$ }
\label{fig:wave1}
\end{subfigure}
~
\begin{subfigure}{0.45\textwidth}
\centering
\begin{tikzpicture}
\begin{axis}[width=1.2\textwidth, height = 0.7\textwidth,ymin=-1,ymax=7,xmin=0,xmax=1,grid=both]
\addplot[mark=none, red!70, thick] table [y=y300,x=x, col sep=comma]{leapfrog_4.txt};
\addplot[mark=none, black, dashed, thick] table [y=y300,x=x, col sep=comma]{conv_4.txt};
\draw[mark=none, fill = gray, opacity=0.2] (axis cs:0.5,-1) rectangle (axis cs:0.7,8);
\draw (50,725) node {$\mathcal{E}^n = 0.7800 \cdot 10^{-5}$};
\end{axis}
\end{tikzpicture}
\caption{$t = 2.930\cdot 10^{-9}$ }
\label{fig:wave2}
\end{subfigure}

\bigskip

\begin{subfigure}{0.45\textwidth}
\centering
\begin{tikzpicture}
\begin{axis}[width=1.2\textwidth, height = 0.7\textwidth,ymin=-1,ymax=7,xmin=0,xmax=1,grid=both]
\addplot[mark=none, red!70, thick] table [y=y500,x=x, col sep=comma]{leapfrog_4.txt};
\addplot[mark=none, black, dashed, thick] table [y=y500,x=x, col sep=comma]{conv_4.txt};
\draw[mark=none, fill = gray, opacity=0.2] (axis cs:0.5,-1) rectangle (axis cs:0.7,8);
\draw (50,725) node { $\mathcal{E}^n = 0.6511 \cdot 10^{-5}$};
\end{axis}
\end{tikzpicture}
\caption{$t = 4.883\cdot 10^{-9}$}
\label{fig:wave3}
\end{subfigure}
~
\begin{subfigure}{0.45\textwidth}
\centering
\begin{tikzpicture}
\begin{axis}[width=1.2\textwidth, height = 0.7\textwidth,ymin=-1,ymax=7,xmin=0,xmax=1,grid=both]
\addplot[mark=none, red!70, thick] table [y=y700,x=x, col sep=comma]{leapfrog_4.txt};
\addplot[mark=none, black, dashed, thick] table [y=y700,x=x, col sep=comma]{conv_4.txt};
\draw[mark=none, fill = gray, opacity=0.2] (axis cs:0.5,-1) rectangle (axis cs:0.7,8);
\draw (50,725) node { $\mathcal{E}^n = 0.6154 \cdot 10^{-5}$};
\end{axis}
\end{tikzpicture}
\caption{$t = 6.836\cdot 10^{-9}$}
\label{fig:wave4}
\end{subfigure}
\caption{Snapshots of the component $h_y$ of the numerical solution restricted to the interval $[0,1]$ at different time steps. The solution of the leapfrog method \eqref{eq:disc1}--\eqref{eq:disc3} is drawn in red while that of the convolution-quadrature method \eqref{eq:disc1}--\eqref{eq:disc2} and \eqref{eq:convsum} is depicted in black. The gray area indicates the location of the dispersive medium.}
\label{fig:wave} 
\end{figure}

\vspace*{-1em}

From the results in Figure~\ref{fig:wave}, one can also recognize the basic physical behavior: In the initial phase, the pulse propagates through air and the total energy of the system is conserved exactly. When impinging on the air-tissue interface, a part of the pulse gets reflected and the rest penetrates into the dispersive medium. Propagation in the medium is substantially slower and, moreover, energy is dissipated according to Lemma~\ref{lem:1}.
We were able to reproduce this energy balance up machine precision. 

\vspace*{-1em}

\section{Summary} \label{sec:5}

We presented two discretization strategies for simulating Maxwell's equation in dispersive media, which were proven to be equivalent for certain classes of problems and to comply with the underlying energy--dissipation structure of the problem. 
The second scheme, which is based on a convolution quadrature approach, is independent of the number of internal states or relaxation times, and can be applied to dispersive media with rather general memory kernels. 
This might become particularly useful also in the context uncertainty quantification.

\vspace*{-1em}

\section*{Acknowledgements}
The authors are grateful for support by the German Research Foundation (DFG) via grants TRR~146 project C03, TRR~154, project C04, and Eg-331/1-1 and through grant Center for Computational Engineering at TU Darmstadt.


\vspace*{-1em}

\begin{thebibliography}{99} 

\bibitem{Bokil14}
Bokil, V. A. and Gibson, N. L.:
\newblock Convergence analysis of Yee schemes for Maxwell's equations in Debye and Lorentz dispersive media.
\newblock Int. J. Numer. Anal. Model. \textbf{11}, 657--687, 2014.

\bibitem{Clegg12}
Clegg, J. and Robinson, M.:
\newblock A genetic algorithm for optimizing multi-pole Debye models of tissue dielectric properties.
\newblock Phys. Med. Biol. \textbf{57}, 6227-43, 2012.

\bibitem{Cohen02}
G.~Cohen.
\newblock Higher-Order Numerical Methods for Transient Wave Equations.
\newblock Springer, 2002.

\bibitem{Debye29}
Debye, P.:
\newblock Polar Molecules.
\newblock Chemical Catalogue Company, New York, 1929.

\bibitem{Gabriel96}
Gabriel S., Lau R. W. and Gabriel C.:
\newblock The dielectric properties of biological tissues: III. Parametric models for the dielectric spectrum of tissues.
\newblock Phys. Med. Biol. \textbf{41}, 2271-93, 1996.

\bibitem{Gandhi93}
Gandhi, Om P. and Gao, B.-Q. and Chen, J.-Y.:
\newblock A frequency-dependent finite-difference time-domain formulation for general dispersive media.
\newblock IEEE Trans. Microw. Theory Tech. \textbf{44-4}, 658-665, 1993.

\bibitem{Egger18}
H.~Egger and B.~Radu.
\newblock A mass-lumped mixed finite element method for {M}axwell's equations.
\newblock {\em arXiv:1810.06243}, 2018.
\newblock to appear in Proceedings of SCEE 2018.

\bibitem{Egger20}
Egger, H., Schmidt, K., Shashkov, V.:
\newblock Multistep and Runge--Kutta convolution quadrature methods for coupled dynamical systems.
\newblock  J. Comput. Appl. Math., 112618, 2020.

\bibitem{Jiao01}
Jiao, D. and Jin,  J.-M.:  
\newblock Time-domain finite-element modeling of dispersive media.
\newblock IEEE Microwave Wireless Components Lett. \textbf{11}, 220-222, 2001.

\bibitem{Jenkinson18}
Jenkinson, M. J. and Banks, J. W.:
\newblock High-order accurate FDTD schemes for dispersive Maxwell’s equations in second-order form using recursive convolutions.
\newblock J. Comput. Appl. Math. \textbf{336}, 192-218, 2018.

\bibitem{Joly03}
P.~Joly.
\newblock Variational methods for time-dependent wave propagation problems.
\newblock In: Topics in Computational Wave Propagation, volume~31 of 
  LNCSE, pages 201--264. Springer, 2003.

\bibitem{Lanteri12}
Lanteri, S. and Scheid C.:
\newblock Convergence of a discontinuous Galerkin scheme for the mixed time-domain Maxwell's equations in dispersive media.
\newblock IMANUM \textbf{33}, 432-459, 2012. 

\bibitem{Li06}
Li, J.:
\newblock Error analysis of finite element methods for 3-D Maxwell’s equations in dispersive media. 
\newblock J. Comput. Appl. Math. \textbf{188}, 107-120. 2006.

\bibitem{Lubich88a}
Lubich, C.:
\newblock Convolution quadrature and discretized operational calculus. I.
\newblock Numer. Math. \textbf{52}, 129-145, 1988.

\bibitem{Lubich88b}
Lubich, C.:
\newblock Convolution quadrature and discretized operational calculus. II.
\newblock Numer. Math. \textbf{52-4}, 413-425, 1988.

\bibitem{Lubich93}
Lubich, C. and Ostermann, A.:
\newblock Runge-Kutta methods for parabolic equations and convolution quadrature.
\newblock Math. Comp. \textbf{60-201}, 105-131, 1993.

\bibitem{Luebbers90}
Luebbers, R., Hunbserger F. P., Kunz, K. S. Standler, R. B., and Schneider, M.:
\newblock A frequency-dependent finite-difference time-domain formulation for dispersive materials. 
\newblock IEEE Trans. Electromag. Compat. \textbf{32}, 222-227, 1990.

\bibitem{Roychowdhury99}
J. Roychowdhury. 
\newblock Reduced-order modeling of time-varying systems. 
\newblock IEEE Trans. Circuits Syst. II \textbf{46}, 1273-1288, 1999.

\bibitem{Schaedle05}
Sch{\"a}dle, A. and L{\'o}pez-Fernandez, M. and Lubich, C.:
\newblock Fast and oblivious convolution quadrature.
\newblock SIAM J. Sci. Comput. \textbf{28}, 421-438, 2006

\bibitem{Shaw10}
Shaw, S.:
\newblock Finite element approximation of Maxwell's equations with Debye memory.
\newblock Adv. Numer. Anal., 923832, 2010.

\end{thebibliography}
\end{document}